\documentclass[12pt,reqno]{amsart}
\usepackage{fullpage}
\allowdisplaybreaks
\usepackage{times}
\usepackage{colonequals}
\usepackage{amsmath,amssymb,amsthm,url}
\usepackage[utf8]{inputenc}
\usepackage[english]{babel}
\usepackage[alphabetic]{amsrefs}
\usepackage{bbm}
\usepackage{enumerate}
\usepackage{bm}
\usepackage{graphicx}
\usepackage{mathrsfs}
\usepackage[colorlinks=true, pdfstartview=FitH, linkcolor=blue, citecolor=blue, urlcolor=blue]{hyperref}
\newtheorem{thm}{Theorem}[section]
\newtheorem{lem}[thm]{Lemma}
\newtheorem{prop}[thm]{Proposition}

\theoremstyle{definition}

\newtheorem{rem}[thm]{Remark}

\newcommand{\bP}{\mathbb{P}}

\newcommand{\F}{\mathbb{F}}
\newcommand{\Fq}{\mathbb{F}_q}

\AtBeginDocument{%
	\def\MR#1{}
}

\begin{document}

\title{Existence of pencils with nonblocking hypersurfaces}
\author{Shamil Asgarli}
\address{Department of Mathematics and Computer Science \\ Santa Clara University \\ 500 El Camino Real \\ USA 95053}
\email{sasgarli@scu.edu}

\author{Dragos Ghioca}
\address{Department of Mathematics \\ University of British Columbia \\ 1984 Mathematics Road \\ Canada V6T 1Z2}
\email{dghioca@math.ubc.ca}

\author{Chi Hoi Yip}
\address{Department of Mathematics \\ University of British Columbia \\ 1984 Mathematics Road \\ Canada V6T 1Z2}
\email{kyleyip@math.ubc.ca}


\maketitle

\begin{abstract}
We prove that there is a pencil of hypersurfaces in $\mathbb{P}^n$ of any given degree over a finite field $\mathbb{F}_q$ such that every $\mathbb{F}_q$-member of the pencil is not blocking with respect to $\mathbb{F}_q$-lines. 
\end{abstract}

\section{Introduction}

Let $\mathcal{P}$ represent some property a given algebraic hypersurface $X\subset \mathbb{P}^n$ may satisfy. For instance, the property $\mathcal{P}$ could be ``is smooth", ``is irreducible", or ``has no rational points." There are multiple perspectives in which a given property $\mathcal{P}$ may hold for a generic hypersurface. When the base field is a finite field $\mathbb{F}_q$, which will be the case for our consideration, there are at least three natural ways to express how a given property $\mathcal{P}$ may be generic: 

\begin{enumerate}
    \item \label{generic-condition-1} ($d$ is fixed) The natural density of hypersurfaces of degree $d$ over $\F_q$ which satisfies $\mathcal{P}$ tends to $1$ as $q\to\infty$, or at least, tends to some function $f(d)$ which is $1-o_d(1)$ as $d \to \infty$.
    \item \label{generic-condition-2} ($q$ is fixed) The natural density of hypersurfaces of degree $d$ which satisfies $\mathcal{P}$ tends to $1$ as $d\to\infty$, or at least, tends to some function $f(q)$ which is $1-o_q(1)$ as $q \to \infty$.
    \item \label{generic-condition-3} ($q,d$ are both fixed) The parameter space of hypersurfaces of degree $d$ has large-dimensional linear spaces whose $\Fq$-points all correspond to hypersurfaces which satisfy $\mathcal{P}$.
\end{enumerate}

The statements~\eqref{generic-condition-1} and \eqref{generic-condition-2} can be viewed as saying that hypersurfaces satisfying $\mathcal{P}$ are abundant on a ``global" level. In contrast, the statement~\eqref{generic-condition-3} is about the ``local" distribution of hypersurfaces satisfying $\mathcal{P}$. When $q$ and $d$ are sufficiently large, the statements~\eqref{generic-condition-1} and~\eqref{generic-condition-2} suggest, but do not directly imply the statement~\eqref{generic-condition-3}. 

When the property $\mathcal{P}$ stands for smoothness, Lang-Weil bounds~\cite{LW54} imply that \eqref{generic-condition-1} holds, and Poonen's theorem~\cite{P04} computes the exact density of smooth hypersurfaces over $\Fq$ and justifies \eqref{generic-condition-2}. For condition~\eqref{generic-condition-3}, the first two authors gave a positive answer to the existence of pencil of smooth hypersurfaces in \cite{AG23} for sufficiently large $q$ when $d,n$ are fixed; the existence of large-dimensional families of smooth hypersurfaces was essentially settled in a subsequent work \cite{AGR23} joint with Reichstein.

In this paper, we will investigate the condition when $\mathcal{P}$ represents the property that a hypersurface is nonblocking. We say that a hypersurface $X\subset\mathbb{P}^n$ is \emph{nonblocking with respect to lines} if there exists an $\Fq$-line $L$ such that $X\cap L$ has no $\Fq$-points. We will drop the part ``with respect to lines" for brevity, and call such a hypersurface \emph{nonblocking}. When $d$ is fixed, and $q \to \infty$, almost all curves are smooth, hence irreducible, by the Lang-Weil bounds. Thus, almost all curves are nonblocking by \cite{AGY23}*{Theorem 1.2}, settling~\eqref{generic-condition-1} in dimension two. In our previous paper \cite{AGY22b} we showed that most plane curves are nonblocking from an arithmetic statistics point of view, thus settling~\eqref{generic-condition-2} in dimension two. The goal of the present paper is to illustrate the abundance of nonblocking hypersurfaces over finite fields by examining \eqref{generic-condition-3}.

The concept of blocking plane curves was formally introduced in \cite{AGY23} with a view toward the rich interplay between finite geometry and algebraic geometry. One of the main tools in the study of blocking sets is to consider an associated algebraic curve (or a variety in general) and study its geometry \cite{SS98}. This paper focuses on the other direction: understanding when the points on a given algebraic variety form a blocking set. Recall that a set of points $B\subseteq \bP^n(\F_q)$ is a \emph{blocking set} (with respect to lines) if every $\F_q$-line meets $B$. A blocking set $B$ is \emph{trivial} if it contains all the $\F_q$-points of a hyperplane defined over $\F_q$, and is otherwise said to be \emph{nontrivial}. One question of particular interest is determining the minimum size of a nontrivial blocking set; we refer to \cite{BSS14} for a recent survey on related topics.

Our main theorem asserts the existence of completely nonblocking pencils. This case precisely corresponds to the case of $1$-dimensional (projective) linear subspaces in the statement~\eqref{generic-condition-3} where the property $\mathcal{P}$ stands for ``is nonblocking".

\begin{thm}\label{thm:main}
Let $n\geq 2$, $d\geq 2$, and $q$ be an arbitrary prime power. There exists a pencil $\mathcal{L}$ of degree $d$ hypersurfaces in $\mathbb{P}^n$ such that every $\Fq$-member of $\mathcal{L}$ is nonblocking.
\end{thm}

We remark that linear systems of hypersurfaces over finite fields have been investigated from a few different perspectives in the literature (see, for example, \cite{Bal07} and \cite{Bal09}). There is also a version of ``simultaneous" Bertini's theorem for a pencil of hypersurfaces over finite fields \cite{AG22FFA}.

\medskip

\textbf{Structure of the paper.} The proof of the main theorem will be separated into two cases, according to whether $n\geq 3$ or $n=2$. In Section~\ref{sect:high-dimensions}, we handle the case $n\geq 3$ by employing a geometric argument that takes advantage of the fact that there exist nonblocking hypersurfaces in $\mathbb{P}^n$ containing a line (see Lemma~\ref{lemma:good-hypersurface}). The proof for the $n=2$ case is more novel and requires a delicate choice of a pencil. In Section~\ref{subsect:past-results}, we briefly discuss how the past results in the literature imply only certain special cases of our main theorem when $n=2$. We then present a detailed proof of Theorem~\ref{thm:main} when $n=2$ in Section~\ref{subsect:proof-for-plane-curves}. Finally, in Section~\ref{sect:efficient}, we discuss how ``efficient" a completely nonblocking pencil of curves can be.

\section{Pencil of nonblocking hypersurfaces in high dimensions}\label{sect:high-dimensions}

The purpose of this section is to prove Theorem~\ref{thm:main} for $n\geq 3$. 

\medskip 

We start by proving the following auxiliary result, which guarantees the existence of a nonblocking hypersurface that contains a fixed $\Fq$-line. 

\begin{lem}\label{lemma:good-hypersurface}
Let $L$ be a fixed $\F_q$-line in $\mathbb{P}^n$ with $n\geq 3$. Given any $d\geq 2$, there exists a hypersurface $X\subset \mathbb{P}^n$ defined over $\Fq$ with degree $d$ such that $L\subset X$ and $X$ is not blocking.
\end{lem}

\begin{proof}
Let $x_0, x_1, ..., x_n$ be the homogeneous coordinates on $\mathbb{P}^n$. Without loss of generality, we can assume that $L = \{x_0=x_1=...=x_{n-2}=0\}$. Let $X$ be a hypersurface defined by the equation $F(x_0, x_1, ..., x_n) = x_0^d + x_1 h(x_2, x_n)$, where $h(x_2, x_n)$ is a homogeneous polynomial of degree $d-1$ so that the specialized polynomial:
$$
F |_{x_2 = x_3 = ... = x_{n-1} = x_0, x_{n} = x_{1}} = x_0^d + x_1 h(x_0, x_1)
$$
has no $\F_q$-roots $[x_0:x_1]$ in $\mathbb{P}^1$. To see why such $h(x_2, x_n)$ exists, we can start with an \emph{irreducible} binary form of degree $d$ in $x_0$ and $x_1$ defined over $\F_q$,
$$
x_0^d + a_{1} x_0^{d-1} x_1 + a_2 x_0^{d-2} x_1^2 + ... + a_{d} x_1^{d}
$$
and re-group the terms,
$$
x_0^d + x_1(a_{1} x^{d-1} + a_2 x^{d-2} y + ... + a_{d} y^{d-1})
$$
and finally set $h(x_2, x_n) = a_{1} x_2^{d-1} + a_2 x_2^{d-2} x_n + ... + a_{d-1} x_2 x_n^{d-2} + a_{d} x_n^{d-1}$.

By construction, the hypersurface $X$ contains the line $L$, since substituting $x_0=x_1=...=x_{n-2}=0$ makes the equation of $F$ identically vanish. On the other hand, we claim that $X$ is not blocking. Indeed, consider the intersection of $X$ with the $\F_q$-line $L_{1}$ given by $\{x_2 = x_3 = ... = x_{n-1} = x_0\}\cap \{x_n=x_1\}$. The intersection $X\cap L_{1}$ is computed by specializing the defining equation $F(x_0, ..., x_n)=0$ of the hypersurface by setting the variables $x_2, ..., x_{n-1}$ equal to $x_0$, and setting the variable $x_n$ equal to $x_1$. By construction, this results in a binary form $x_0^d + x_1 h(x_0, x_1)$ which has no $\F_q$-roots in $\mathbb{P}^1$. In particular, $X\cap L_1$ has no $\F_q$-points, and therefore $X$ is not blocking.  \end{proof}

As another ingredient in our proof, we will rely on the following lemma regarding interpolation in algebraic geometry. We denote by $V_d$ the vector space of homogeneous forms of degree $d$ in $n+1$ variables $x_0, x_1, ..., x_n$. The projective space $\mathbb{P}(V_d)$ paramaterizes hypersurfaces of degree $d$ in $\mathbb{P}^n$.

\begin{lem}\label{lem:independence}
Fix a finite field $\mathbb{F}_q$, and consider any $k$ distinct  $\overline{\F_q}$-points $P_1, P_2, \ldots, P_k$ in $\mathbb{P}^n$. Let $W\subseteq V_d$ be the subspace (over $\overline{\Fq}$) corresponding to the hypersurfaces of degree $d$ passing through $P_1, P_2, \ldots, P_k$. If $d\geq k-1$, then $W$ has codimension $k$.  
\end{lem}
\begin{proof}
The proof of this lemma has already appeared in the special case of plane curves $(n=2)$ in our previous work~\cite{AGY22b}*{Proposition 3.1}. The same proof extends to the hypersurface case by replacing every occurrence of the word ``line" in that proof with the word ``hyperplane." We also mention that the result is known to the experts (see \cite{P04}*{Lemma 2.1} for proof using the cohomological language). 
\end{proof}

We now proceed to the proof of the main theorem for $n\geq 3$.

\begin{proof}[Proof of Theorem~\ref{thm:main} for $n \geq 3$]
Fix an $\F_q$-line $L$, say $L = \{x_0=x_1=...=x_{n-2}=0\}$. We have a subspace of $V_d$ given by:
$$
V_{d, L} = \{ f \in V_d \ | \ f \text{ identically vanishes on } L\}.
$$
In other words, the polynomials in $V_{d, L}$ correspond to hypersurfaces which contain the $\mathbb{F}_q$-line $L$. The codimension of $V_{d, L}$ inside $V_{d}$ is exactly $d+1$. This is because a homogeneous polynomial $f$ of degree $d$ vanishes along $L = \{x_0=x_1=...=x_{n-2}=0\}$ if and only if the $d+1$ coefficients in front of the monomials $x_{n-1}^i x_n^{d-i}$ (for $i=0, 1, ..., d$) all vanish. These $d+1$ coefficients are coordinates in the parameter space, so we obtain that $\operatorname{codim}(V_{d, L})=d+1$.

Now, let $Q_1\in L(\mathbb{F}_{q^d})$ such that $Q_1\notin L(\mathbb{F}_{q^r})$ for $r<d$. Consider the orbit of $Q$ with its conjugates under the Frobenius map $[x_0:...:x_n]\mapsto [x_0^q:...:x_n^q]$. This orbit forms a set $S=\{Q_1, Q_2, ..., Q_{d}\}$ of $d$ distinct points invariant under the Frobenius action. By Lemma~\ref{lem:independence}, we see that passing through these $d$ points $Q_1, ..., Q_d$ imposes linearly independent conditions in the space $V_d$. Thus, the vector subspace,
$$
W =  \{ f \in V_d \ | \ f \text{ vanishes at all of } Q_1, Q_2, ..., Q_d \}
$$
has codimension $d$ inside $V_d$. Note that $W$ is a subspace defined over $\Fq$, because $\{Q_1, Q_2, ..., Q_d\}$ forms a Galois orbit. It is also clear that $V_{d, L} \subseteq W$. Since $\dim(W) = \dim(V_{d, L})+1$, we see that $V_{d, L}$ is an $\mathbb{F}_q$-hyperplane inside $W$. After projectivizing, $\mathbb{P}(V_{d, L})$ is a hyperplane inside $\mathbb{P}(W)$. 

By Lemma~\ref{lemma:good-hypersurface}, there exists a nonblocking hypersurface $X\in V_{d, L}$ defined over $\F_q$. Consider the pencil $\mathcal{L}=\mathbb{P}^1$ spanned by the point $X$ and other $Y\in W \setminus V_{d, L}$ such that $Y$ is also defined over $\F_q$. By construction, $\mathcal{L}$ lies entirely inside $\mathbb{P}(W)$ and intersects $\mathbb{P}(V_{d, L})$ in exactly the point $X$. We claim that each $q+1$ distinct members of $\mathcal{L}$ is nonblocking. This assertion is true for the special hypersurface $X$ by construction. The other $q$ distinct $\mathbb{F}_q$-members of $\mathcal{L}$ are nonblocking hypersurfaces, because they intersect the $\mathbb{F}_q$-line $L = \{x_0=x_1=...=x_{n-2}=0\}$ in no $\mathbb{F}_q$-points. Indeed, they intersect $L$ in the non-$\mathbb{F}_q$-points $Q_1, Q_2, ..., Q_{d}$ and there are no other points along $L$ due to Bezout's theorem. We can apply Bezout's theorem because these $q$ members of the pencil (other than $X$) do not contain the line $L$, since they are not in $\mathbb{P}(V_{d, L})$.
\end{proof}

\section{Pencil of nonblocking plane curves}\label{sect:plane-curves}

This section will prove Theorem~\ref{thm:main} in the case $n=2$. Note that Lemma~\ref{lemma:good-hypersurface} fails trivially in this case, so we need to modify our approach. 

\subsection{Comparison with past results.} \label{subsect:past-results} As we will see in the next subsection, the proof of Theorem~\ref{thm:main} is considerably more intricate for the plane curve case compared to the higher-dimensional case. We give additional context to the difficulty of this problem. In particular, we explain how past results in the literature about blocking curves can only prove Theorem~\ref{thm:main} in certain special cases. 

\medskip

\textbf{Special Case 1.} Suppose $\operatorname{char}(\Fq)\neq 3$ and $q>d^6$. There exists a smooth pencil $\mathcal{L}$ with degree $d$ over $\F_q$ whenever $\operatorname{char}(\Fq)\neq 3$ by \cite{AGR23}*{Theorem 2}, and all $\Fq$-members in such a pencil are nonblocking curves provided that $q>d^6$ \cite{AGY23}*{Theorem 1.2}.

\medskip 

\textbf{Special Case 2.} Suppose $\operatorname{char}(\Fq)= 3$ and $q>C d^{12}$ for some absolute constant $C$. There exists a smooth pencil $\mathcal{L}$ with degree $d$ over $\F_q$ \cite{AG23}*{Theorem 1.3} provided that $q>Cd^{12}$, and all $\Fq$-members in such a pencil are nonblocking curves \cite{AGY23}*{Theorem 1.2}.

\medskip 

\textbf{Special Case 3.} Suppose $d\geq 3(q-1)$. We can partition $\bP^2(\F_q)$ into $q$ sets with size $q$ and a set with size $q+1$ such that not all $q+1$ points are collinear. Then none of these $q+1$ sets are blocking. By applying \cite{AGY22c}*{Proposition 2.1}, we obtain a pencil $\mathcal{L}$ such that the sets of $\Fq$-points on the $q+1$ members induce the same partition. Thus, all $\Fq$-members in such a pencil $\mathcal{L}$ are nonblocking curves.

\medskip 

We mentioned in the introduction that statements~\eqref{generic-condition-1} and \eqref{generic-condition-2} suggest that statement~\eqref{generic-condition-3} is likely to be true. The special cases above can be regarded as consequences of statements~\eqref{generic-condition-1} and \eqref{generic-condition-2}. However, a proof of Theorem~\ref{thm:main} in full generality cannot rely on statements~\eqref{generic-condition-1} and \eqref{generic-condition-2}.

\medskip 

\subsection{Proof of Theorem~\ref{thm:main} for $n=2$}\label{subsect:proof-for-plane-curves} We will work with a fixed degree $d\geq 2$ over a finite field $\Fq$. The key idea is to find a pencil such that each $\Fq$-member is irreducible and has at most $q+2$ distinct $\Fq$-points. We break the proof into several steps.

\medskip

\textbf{Step 1} (Construction of the pencil). Consider the map $f_d \colon \Fq\to\Fq$ given by $f_d(x)=x^{2d-1}-x^{d-1}$. Since $f_d(0)=f_d(1)=0$, the map $f_d$ is not injective and therefore not surjective. Let $\beta\in\F_q^{\ast}$ be any element not in the image of $f_d$. Let $g=1/\beta\in\F_q^{\ast}$, and consider the pencil $\langle F_d, G_d\rangle$ where:
$$
F_d(x,y,z) = y^d - z^d - z^{d-1} x, \quad \quad \text{and} \quad \quad G_d(x,y,z) = z^d + g y^{d-1} x.
$$
We claim that the pencil $\mathcal{L}=\langle F_d, G_d\rangle$ satisfies the desired properties of Theorem~\ref{thm:main}, that is, all the $q+1$ curves defined over $\Fq$ in $\mathcal{L}$ are nonblocking.

\medskip 

\textbf{Step 2} (Reduction to the low degree case).
For each $[s:t] \in \bP^1(\F_q)$, let $C_{d,[s:t]}$ denote the curve $sF_d-tG_d=0$. To check that every $\Fq$-member $\mathcal{L}$ is a nonblocking curve, let us explain why considering the case when $d\leq q$ suffices.  Recall that a curve $C$ is nonblocking if and only if $C(\F_q)$ is not a blocking set. Thus, it suffices to show $C_{d,[s:t]}(\F_q)=C_{q+d-1,[s:t]}(\F_q)$ for each $[s:t] \in \bP^1(\F_q)$. This follows from the shape of $F_d$ and $G_d$ and the identity $b^{r+q-1}=b^r$ that holds for any $r\geq 1$ and $b\in\Fq$. If we can show that the conclusion holds for the pencil $\langle F_d, G_d\rangle$ of degree $d$, then it will also hold for the pencil $\langle F_{q+d-1}, G_{q+d-1}\rangle$ of degree $d+q-1$. Consequently, we may assume $d\leq q$ for the rest of the proof.

\medskip 

\textbf{Step 3} (Irreducibility of the curves). For every $[s:t]\in\bP^1(\Fq)$, we claim that the polynomial 
\begin{equation}\label{irreducibility-1}
sF_d - t G_d = s(y^d - z^d + z^{d-1}x) - t(z^d + g y^{d-1}x) 
\end{equation}
is absolutely irreducible, that is, irreducible in $\overline{\F_q}[x,y,z]$. When $s=0$ or $t=0$, the result easily follows. We will assume $s\neq 0$ and $t\neq 0$. After scaling by $1/s$, it suffices to show that the polynomial
\begin{equation}\label{irreducibility-2}
y^d - z^d + z^{d-1}x - t(z^d + g y^{d-1}x) 
\end{equation}
is irreducible in $\overline{\F_q}[x,y,z]$ for every $t\in\Fq^{\ast}$. Assume, to the contrary, that the polynomial in \eqref{irreducibility-2} splits nontrivially. Since it has degree $1$ in $x$, we must have a factor $z-ay$ (using homogeneity of the polynomial in variables $y$ and $z$) for some $a\in\overline{\Fq}$. Substituting $z=ay$ into the polynomial \eqref{irreducibility-2} and collecting  coefficients of $xy^{d-1}$ and $y^d$, respectively, we obtain two relations:
\begin{align}
a^{d-1} &= t g \label{t-relation-1} \\
1-a^d - t a^d &= 0. \label{t-relation-2}
\end{align}
From \eqref{t-relation-2}, we get $a^d = \frac{1}{1+t}$ which implies $a^d \in\Fq^{\ast}$ (note $t=-1$ is not possible in view of \eqref{t-relation-2}). Combining $a^d \in\Fq^{\ast}$ with \eqref{t-relation-1}, we deduce that $a\in\Fq^{\ast}$. On the other hand, \eqref{t-relation-2} implies
\begin{equation}\label{t-relation-3}
t = \frac{1-a^d}{a^d}.
\end{equation}
Combining \eqref{t-relation-1} and \eqref{t-relation-3}, we obtain
$$
a^{d-1} = \frac{1-a^d}{a^d} \cdot g \ \ \Longrightarrow \ \ \frac{1}{g} = \frac{1-a^d}{a^{2d-1}} = \left(\frac{1}{a}\right)^{2d-1} - \left(\frac{1}{a}\right)^{d-1}.
$$
Therefore, $\frac{1}{g} = \beta$ is in the image of the function $f_d(x) = x^{2d-1}-x^{d-1}$. This last sentence contradicts the choice of $\beta$, and we conclude the (absolute) irreducibility of the polynomial in \eqref{irreducibility-1}. We have established that each $\Fq$-member of the pencil $\langle F_d, G_d\rangle$ is an irreducible curve.

\medskip 

\textbf{Step 4} (Counting the number of $\Fq$-points). Let $[s:t]\in\bP^1(\Fq)$. We claim that there are at most $q+2$ distinct $\Fq$-points on the curve defined by:
$$
C_{d,[s:t]}:sF_d - t G_d = s\cdot (y^d - z^d + z^{d-1}x) - t\cdot (z^d + g y^{d-1}x) = 0. 
$$
We will now analyze several cases, depending on whether or not $s$ or $t$ are zero.

When $t=0$, we have the curve $C_0\colonequals C_{d, [1:0]}$ defined by $y^d-z^d+z^{d-1}x=0$. We claim that $C_0$ has $q+1$ distinct $\F_q$-points. To count the number of $\Fq$-points of $C_0$, we consider two cases:  

\textbf{Case 1.} $z\neq 0$. In this case, for any $(y, z)\in\Fq\times\Fq^{\ast}$, we can uniquely solve for $x\in\Fq$. For any $r\in\Fq^{\ast}$, the pair $(ry, rz)$ results in $rx$, leading to the same point $[rx:ry:rz]=[x:y:z]$ in $\mathbb{P}^2$. Thus, there are $\frac{q\cdot (q-1)}{q-1}=q$ distinct $\Fq$-points on $C_0$ with $z\neq 0$. 

\textbf{Case 2.} $z=0$. In this case, $y=0$ which means that $[x:y:z]=[1:0:0]$ is one additional $\Fq$-point in $C_0$. 

\medskip

When $s=0$, we have the curve $C_{\infty}\colonequals C_{d, [0:1]}$ defined by $z^d+g y^{d-1}x=0$. The similar analysis applies: when $y\neq 0$, we can solve uniquely for $x\in\Fq$, giving us a total of $\frac{q\cdot (q-1)}{q-1}=q$ points. When $y=0$, we get an additional $\Fq$-point $[1:0:0]$ on $C_{\infty}$. Thus, $C_{\infty}$ has $q+1$ distinct $\Fq$-points as well. 

\medskip

Next, we focus on the case when $s\neq 0$ and $t\neq 0$. After scaling, it is enough to work with the curves defined by $C_{d, [1:t]}$, which we denote by $C_{t}$ for simplicity. We have,
\begin{equation}
C_{t} \colon \ y^d - z^d + z^{d-1}x - t(z^d + g y^{d-1}x) = 0 
\end{equation}
where $t\in\Fq^{\ast}$.  We rewrite,
\begin{equation}\label{curve:bounding-rat-points}
C_{t} \colon \ y^d - z^d - t z^d + (z^{d-1} - t g y^{d-1})x = 0. 
\end{equation}

To count the number of $\Fq$-points $[x:y:z]$ on $C_{t}$, we consider two cases:

\textbf{Case 1.} $z^{d-1} -
t g y^{d-1} \ne 0$. In this case, the number of possible pairs is $(y,z)\in \Fq\times \Fq$ is at most $q^2 - 1$. For each such pair, we can solve for $x$ uniquely in \eqref{curve:bounding-rat-points}. Since $(ry, rz)$ results in $rx$ for any $r\in\Fq^{\ast}$, which corresponds to $[rx:ry:rz]=[x:y:z]$ in $\mathbb{P}^2$, the number of $\Fq$-points on $C_{t}$ in this case is at most $\frac{q^2-1}{q-1}=q+1$. 

\textbf{Case 2.} $z^{d-1} -
t g y^{d-1} = 0$. In this case, we must also have $y^d-z^d-tz^d=0$. The only additional $\Fq$-point we get on $C_t$ is $[1:0:0]$. Indeed, the analysis in \textbf{Step 3} shows that the only solution to the system of two  equations:
$$
y^d - z^d - t z^d = 0, \quad \text{and} \quad z^{d-1} - t g y^{d-1} = 0
$$
in $\overline{\F_q}\times \overline{\F_q}$ is $(y,z)=(0,0)$, for otherwise $C_{t}$ is not geometrically irreducible.

We conclude that $C_{t}$ has at most $q+2$ distinct $\Fq$-points for every $t\in\Fq^{\ast}$.

\medskip

\textbf{Step 5} (Conclusion). Let $[s:t] \in \bP^1(\F_q)$. We have shown that $C_{d,[s:t]}$ is geometrically irreducible and $|C_{d,[s:t]}(\F_q)| \leq q+2$. For the sake of contradiction, suppose that $C_{d,[s:t]}$ is blocking; then $C_{d,[s:t]}(\F_q)$ must be a trivial blocking set, as otherwise 
$|C_{d,[s:t]}(\F_q)| \geq q+\sqrt{q}+1$ \cite{B71}. It follows that  $C_{d,[s:t]}$ contains all the $q+1$ distinct $\F_q$-points of a line $L$ defined over $\F_q$. However, since $C_{d,[s:t]}$ is geometrically irreducible, $C_{d,[s:t]} \cap L$ has at most $d \leq q$ intersection points by B\'ezout's theorem, a contradiction. This proves that all curves defined over $\Fq$ in the pencil $\mathcal{L}=\langle F_d, G_d\rangle$ are nonblocking.

\begin{rem} If we allow the pencil to have at most one blocking curve, the problem becomes much easier. Indeed, for each $d \geq 2$ and $q$, we can give an explicit construction of a ``near miss" nonblocking pencil as follows. Let $h(t)\in\Fq[t]$ be an irreducible polynomial of degree $d$. Consider the pencil $\mathcal{L}=\langle F, G\rangle$ given by,
$$
F(x,y,z)= x^d \quad\quad \text{and} \quad\quad G(x,y,z)=z^d h(y/z). 
$$
Then the $\Fq$-member corresponding to $F=0$ is trivially blocking, but the other $q$ distinct $\Fq$-members of $\mathcal{L}$ are not blocking (since their intersection with the line $x=0$ has no $\Fq$-points). 
\end{rem}

\section{Efficient nonblocking pencils} \label{sect:efficient}

We have proved in Theorem~\ref{thm:main} the existence of a pencil of plane curves whose $\Fq$-members are nonblocking, that is, every $\Fq$-member admits a \emph{skew} $\Fq$-line (namely, a line which meets the curve at no $\Fq$-points). It is natural to ask how many skew $\Fq$-lines must be present to ensure the pencil is completely nonblocking. It is impossible to have two $\Fq$-lines $L_1$ and $L_2$ such that every $\Fq$-member is skew to either $L_1$ or $L_2$; indeed, the intersection point $P\in L_1\cap L_2$ would be contained in some $\Fq$-member of the pencil. Next, we show that it is possible to have three $\Fq$-lines $L_1, L_2$, and $L_3$ such that every $\Fq$-member is skew to one of $L_1, L_2$, and $L_3$. In other words, for the ``most efficient" pencil of nonblocking curves, three lines are sufficient to witness that all curves in the pencil are nonblocking. 

We begin with the following criterion for a Fermat-type curve to be nonblocking.

\begin{lem}\label{lem:Fermat}
Assume the characteristic of the field $\Fq$ is not $2$. If $(q-1)/\gcd(d,q-1)$ is odd and $a,b,c\in\Fq^{\ast}$, then the curve $C$ defined by 
$$ax^d+by^d+cz^d=0$$
is nonblocking. Moreover, one of the three lines $x=0$, $y=0$, $z=0$ is a skew line to $C$.
\end{lem}
\begin{proof}
Let $d'=\gcd(q-1,d)$. Note $d$-th powers in $\F_q^*$ are essentially $d'$-th powers in $\F_q^*$. If the curve $C$ meets $x=0$, then $by^d+cz^d=0$, which implies that $-b/c$ is a $d'$-th power. Similarly, if the curve $C$ meets both $y=0$ and $z=0$, then $-c/a$ and $-a/b$ are also 
$d'$-th powers. In particular, $-1=(-a/b) \cdot (-b/c) \cdot (-c/b)$ is a $d'$-th power, that is, $(-1)^{(q-1)/d'}=1$, contradicting our assumption that $(q-1)/d'$ is odd and that ${\rm char} (\Fq)\ne 2$.
\end{proof}

\begin{prop}
Assume the characteristic of the field $\Fq$ is not $2$. If $\gcd(q-1,d) \geq 3$ and $(q-1)/\gcd(q-1,d)$ is odd, then there exists a pencil of nonblocking curves of degree $d$ over $\F_q$ witnessed by the lines $x=0,y=0$ and $z=0$.
\end{prop}
\begin{proof}
Let $d'=\gcd(q-1,d)$. Since $(q-1)/d'$ is odd, it follows that $-1$ is not a $d'$-th power in $\F_q^*$. Since $d' \geq 3$, we can always pick an element $r$ in $\F_q^*$ such that both $r$ and $-r$ are \emph{not} $d'$-th powers in $\F_q^*$.

Let $F=x^d+y^d$, and $G=y^d+rz^d$. Since $-1$ is not a $d'$-th power, the $\Fq$-line $z=0$ is skew to the curve $\{F=0\}$. Similarly, since $-r$ is not a $d'$-th power, the line $x=0$ is skew to the curve $\{G=0\}$. In particular, these two curves $\{F=0\}$ and $\{G=0\}$ are nonblocking.  Moreoever, as $r$ is not a $d'$-th power, the curve $\{F-G=0\}$ defined by the polynomial $(x^d+y^d)-(y^d+rz^d)=x^d-rz^d$ is also nonblocking, since it admits $y=0$ as a skew line. In fact, each of these $3$ curves only has one $\F_q$-point. 

Consider the pencil $\mathcal{L}=\langle F, G\rangle$. Other than the $3$ special $\Fq$-members mentioned above, all other curves in $\mathcal{L}$ have the form $ax^d+by^d+cz^d=0$, where $a,b,c$ are nonzero. We can then apply Lemma~\ref{lem:Fermat} to get the desired conclusion. 
\end{proof}

\begin{rem}
The above construction does not generalize to all $q$ and $d \geq 2$. Indeed, when $q$ is a square, the Hermitian curve 
$x^{\sqrt{q}+1} + y^{\sqrt{q}+1} + z^{\sqrt{q}+1} = 0 
$
is blocking \cite{AGY23}*{Example 1.5}. More generally, one can construct a family of Fermat-type curves which are Frobenius nonclassical and blocking \cite{AGY23}*{Section 4}.
\end{rem}

\section*{Acknowledgements}
The second author is supported by an NSERC Discovery grant.

\bibliographystyle{alpha}
\bibliography{biblio.bib}

\end{document}